\documentclass[reqno]{amsart}

\usepackage{amsmath}
\usepackage{amsthm}
\usepackage{amsfonts}
\usepackage[latin1]{inputenc}
\usepackage{color}

\newtheorem{thm}{Theorem}[section]
\newtheorem{lem}[thm]{Lemma}

\newtheorem{prop}[thm]{Proposition}

\DeclareMathAlphabet{\mathpzc}{OT1}{pzc}{m}{it}

\numberwithin{equation}{section}

\newcommand{\R}{\mathbb{R}}

\newcommand{\N}{\mathbb{N}}
\newcommand{\Z}{\mathbb{Z}}
\newcommand{\C}{\mathbb{C}}

\newcommand{\sph}{\mathbb{S}^1}

\newcommand{\B}{\mathbb{B}}

\newcommand{\Om}{\Omega}
\newcommand{\Omi}{\Omega^{\text{i}}}
\newcommand{\Omo}{\Omega^{\text{o}}}

\newcommand{\rd}{\mathrm{d}}

\newcommand{\bqn}{\begin{equation}}
\newcommand{\eqn}{\end{equation}}
\newcommand{\bqnn}{\begin{equation*}}
\newcommand{\eqnn}{\end{equation*}}
\newcommand{\bear}{\begin{eqnarray}} 
\newcommand{\eear}{\end{eqnarray}} 
\newcommand{\bean}{\begin{eqnarray*}} 
\newcommand{\eean}{\end{eqnarray*}} 
\newcommand{\bs}{\begin{split}}
\newcommand{\es}{\end{split}}

\newcommand{\ii}{\text{i}}
\newcommand{\oo}{\text{o}}

\newcommand{\dhr}{\mathrel{\lhook\joinrel\relbar\kern-.8ex\joinrel\lhook\joinrel\rightarrow}}

\begin{document}

\title[Two-Phase Flow with Coriolis Effects]{Two-Phase Flow in Rotating Hele-Shaw Cells with Coriolis Effects}

\author{Joachim Escher}
\address{Leibniz Universit\"at Hannover\\ Institut f\" ur Angewandte Mathematik \\ Welfengarten 1 \\ D--30167 Hannover\\ Germany}
\email{escher@ifam.uni-hannover.de}

\author{Patrick Guidotti}
\address{University of California, Irvine\\
Department of Mathematics\\
340 Rowland Hall\\
Irvine, CA 92697-3875\\ USA }
\email{gpatrick@math.uci.edu}

\author{Christoph Walker}
\address{Leibniz Universit\"at Hannover\\ Institut f\" ur Angewandte Mathematik \\ Welfengarten 1 \\ D--30167 Hannover\\ Germany}
\email{walker@ifam.uni-hannover.de}

\begin{abstract}
The free boundary problem of a two phase flow in a rotating Hele-Shaw cell with Coriolis effects is studied. 
Existence and uniqueness of solutions near spheres is established, and the asymptotic stability and instability 
of the trivial solution is characterized in dependence on the fluid densities.
\end{abstract}

\subjclass[2010]{}

\maketitle


\section{Introduction}

The motion of one or two fluids confined to a narrow gap between two parallel plates is 
an interesting problem with a long history. It is the classical set up of the so-called 
Hele-Shaw problem \cite{HS1898}. 
It is well-known that instability driven pattern formation such as fingering can occur 
under appropriate assumptions on the viscosity of the fluids \cite{ST58} and on the 
surface tension at their interface. There is a vast empirical and theoretical literature 
concerning this classical problem. From the purely mathematical point of view the problem 
has been extensively studied in its original formulation as a one \cite{DR84,CP93,ES97b,FT03} 
and as a two-phase \cite{YT11} problem and, more recently, also {for one fluid} in the case of rotating 
plates \cite{EEM11}. The focus is here on the two-phase problem with rotating plates and 
including the effects of Coriolis forces. This setting has recently been considered in 
the physical literature by a variety of authors \cite{WC05,GBM07,AlvarezGadelhaMiranda}. It first appeared in 
\cite{Schw89} where the effects of rotation were introduced in a ad-hoc fashion into 
Darcy's law for the one-phase problem
$$
{\nabla p=-\frac{12\eta}{b^2}\vec{v}+\varrho\omega ^2 \vec{x}+2\varrho\omega\vec{z}\times \vec{v}\, , }
$$
where $\vec{z}$ is the axis of rotation, $b$ is the plate spacing, $\vec{x}$ is the two-dimensional position vector, $\omega$ is the (constant) angular velocity of the plates, whereas $\vec{v}$, $p$, 
$\eta$, and ${\varrho}$ are the fluid's velocity, pressure, viscosity, and density, 
respectively. The last term accounts for Coriolis' force. {While most of the literature hitherto neglected this force, recent studies performed a model derivation 
for a one-phase Hele-Shaw type model with Coriolis force by means of a standard gap 
averaging technique starting from Navier-Stokes' equations, cf. \cite{SchwR04,WC05}. The authors of \cite{WC05} observe  that the effects 
due to the Coriolis term in their equations can be larger than the inertial terms typically 
neglected in small Reynolds number type reductions. 
They also point out the fact that simpler ad-hoc models used earlier, while qualitatively similar, 
{do} eventually lead to a different and inaccurate prediction of the growth rate of 
the unstable modes involved in the fingering phenomenon.}

More recently, the {effects of a Coriolis term} 
on the fingering patterns in the rotating two-phase Hele-Shaw problem have been studied in 
\cite{GBM07,AlvarezGadelhaMiranda} where the formal linear stability analyses of \cite{Schw89,WC05} for the 
one-phase problem are extended to the weakly nonlinear case by the use of formal expansions 
and to the fully nonlinear regime by means of numerical simulations. Previous numerical and 
experimental results included \cite{SchwR04,CCWT05}, whereas the practical relevance of the 
problem is attested by a number of publications cited in \cite{AlvarezGadelhaMiranda} for that very purpose.

In this paper local existence of a unique classical solution for nearly circular initial 
interfaces is {established for the} general 
rotating two-phase Hele-Shaw problem with Coriolis effects in the formulation proposed in 
\cite{GBM07,AlvarezGadelhaMiranda}, which is, in turn, based on a derivation similar to that of \cite{WC05} 
for the single phase case. It leads to the following generalized 
Darcy's law
$$
 \nabla P_j=-\alpha _j\vec{v}_j+\beta _j\vec{z}\times \vec{v}_j\, ,
$$
for the velocities of the inner and outer fluids ($j=\ii,\oo$), for constants $\alpha_j$ 
and $\beta _j$ {defined below} and for the pressure $P_j$ related to the hydrostatic 
pressure $p_j$ via
$$
 P_j=p_j-\frac{\varrho _j\omega ^2}{2}|x|^2\, .
$$
The use of the {\em unusual} pressure $P_j$ yields some useful simplifications as will 
soon become apparent.
Rigorous stability results are also obtained on near circularity assumptions {for the} initial 
data by resorting to a general principle of linearized stability. {While the circular steady state 
is exponentially asymptotically stable\footnote{{Our analytical approach reveals a precise description 
of the exponential decay in terms of the physical parameters, in particular of the Coriolis terms, 
see Section~\ref{stable}.}} when the outer fluid is denser, it becomes unstable {if this relation is reversed.}} 
The main techniques used are {a transformation} of the {original} free boundary problem to a fixed domain one, 
{a} decoupling of the system and {a} reduction to a single nonlinear and nonlocal evolution 
for the interface separating the two fluids, and parabolic optimal regularity results in little H\"older spaces.

Related previous mathematical results include the existence, stability, and bifurcation 
analysis of \cite{EEM11} in the rotating one-phase problem and the global existence and 
stability of smooth solutions obtained by \cite{YT11} for nearly circular initial interfaces 
in the non-rotating two-phase case.

\section{Governing Equations and Main Results}

We first give a short {justification for} the mathematical formulation of the physical problem of a rotating Hele-Shaw cell 
including Coriolis effects{. F}urther details regarding the modeling can be found in \cite{AlvarezGadelhaMiranda,GBM07}. 
We {then} end this section with {the} main results on existence of solutions to the corresponding governing equations 
and stability properties of the trivial solution. 

\subsection{Governing Equations}

Consider a circular Hele-Shaw cell of radius $R\ge 2$ and very small gap width $b>0$, rotating clockwise around the $z$-axis with  
a constant angular velocity \mbox{$\omega>0$}. The cell is assumed to contain two immiscible, incompressible, viscous fluids with 
densities \mbox{$\varrho_j>0$} and viscosities $\eta_j>0$, where $j=$i labels the inner and $j=$o the outer fluid{, 
respectively}. The surface tension between the two fluids is {given by }$\sigma>0$. The rotating coordinate system is 
defined in such a way that its origin is located at the cell center and that rotation is perpendicular to the vector 
$\vec{z}=(0,0,1)$. Let $\Om^j(t)$ denote the region {of space} occupied by fluid $j\in\{$i,o$\}$ at time $t$, and 
let $\Gamma(t)$ {denote} the sharp interface separating the two fluids. The unit normal vector $\nu_{\Gamma(t)}$ 
{to} $\Gamma(t)$ {is assumed to point} from $\Omi(t)$ to $\Omo(t)$.

The basic hydrodynamic equation of the system is a generalized {\em Darcy's law} relating the pressure fields 
$P_j=P_j(t)$ {to} the two-dimensional gap-averaged flow velocities $\vec{v}_j=\vec{v}_j(t)$ {through}
\bqn\label{1}
\nabla P_j=-\alpha_j\vec{v}_j+\beta_j(\vec{z}\times\vec{v_j})\quad\text{in}\quad \Om^j(t)\ .
\eqn
The numbers $$\alpha_j:=\frac{12\eta_j}{b^2} E_j\, ,\qquad \beta_j:=\frac{12\eta_j}{b^2} F_j$$
depend linearly on the Coriolis force terms $E_j> 0$ and $F_j\ge 0$, respectively{. The latter, in turn,} 
depend on the rotational Reynolds number $Re_j=\varrho_j \omega b^2/12\eta_j$ (see \cite{AlvarezGadelhaMiranda}). 
If Coriolis effects are neglected, then $E_j=1$ and $F_j=0$, so \eqref{1} reduces to the usual Darcy's law. Throughout 
this paper we shall assume that $\alpha_j>0$ and $\beta_j\ge 0$. 
Incompressibility of the fluids is expressed by
\bqn\label{2}
\mathrm{div}\vec{v}_j=0\quad\text{in}\quad \Om^j(t)\, ,
\eqn
while the interface dynamics is governed by the normal stress balance
\bqn\label{3}
P_\ii-P_\oo=\sigma \kappa_{\Gamma(t)}+(\gamma_\oo-\gamma_\ii)\vert x\vert^2 \quad\text{on}\quad \Gamma(t)\, ,
\eqn
and the kinematic boundary condition
\bqn\label{4}
V=\vec{v}_\ii\cdot\nu_\Gamma =\vec{v}_\oo\cdot\nu_\Gamma \quad\text{on}\quad \Gamma(t)\ .
\eqn
Here, $\gamma_j:=\varrho_j\omega^2/2$ {whereas} $\kappa_{\Gamma(t)}$ denotes the curvature of $\Gamma(t)$ 
{which is }taken to be positive if $\Omi(t)$ is convex. {The normal velocity of the interface is given by $V=V(t)$.}
On the outer boundary of the cell {a} no-slip condition
\bqn\label{5}
\vec{v}_\oo\cdot x =0 \quad\text{on}\quad [\vert x\vert=R]
\eqn
is imposed.\\

Observe that the {inner and outer fluids'} masses are conserved since, by Reynold's transport theorem, {one has that}
$$
\frac{\rd}{\rd t}\int_{\Omo(t)}\,\rd x=-\int_{\Gamma(t)}\vec{v}_\oo\cdot\nu_\Gamma\, \rd \sigma +\int_{[\vert x\vert=R]} \vec{v}_\oo\cdot \frac{x}{R}\,\rd \sigma
$$
and
$$
\frac{\rd}{\rd t}\int_{\Omi(t)}\,\rd x=\int_{\Gamma(t)}\vec{v}_\ii\cdot\nu_\Gamma\, \rd \sigma \, ,
$$
which are both zero in view of Gauss' Theorem and \eqref{2}. Since
$\mathrm{rot}(\vec{z}\times\vec{v}_j)=(0,0,\mathrm{div} \vec{v_j})$ we also note that the flows are irrotational in the 
bulk, i.e. $\mathrm{rot}\vec{v}_j=0$ in $\Om^j(t)$, due to \eqref{1} and \eqref{2}. 

\subsection{The System in Terms of $P_j$}
\hspace{0.1pt}
{The} governing equations in terms of $P_j$ {are obtained by taking }the divergence on both 
sides of \eqref{1} {which yields}
\bqnn
\Delta P_j=0 \quad\text{in}\quad \Om^j(t)\, ,
\eqnn
{making use of } \eqref{2} and {observing }that $\mathrm{rot}\vec{v}_j=0$ implies 
$\mathrm{div}(\vec{z}\times\vec{v}_j)=0$.
As for the kinematic boundary condition~\eqref{4}, we solve \eqref{1} for $\vec{v}_j$ {to get}
$$
\vec{v}_j=\frac{1}{\vert \Theta_j\vert^2}\big(-\alpha_j\nabla P_j-\beta_j(\vec{z}\times\nabla P_j)\big)
$$
with complex numbers $\Theta_j:=\alpha_j +i\beta_j$. Then, defining the {tangent} vector $\tau_\Gamma :=-(\vec{z}\times\nu_\Gamma)$ 
{to}~$\Gamma$, we obtain
$$
\vec{v}_j\cdot \nu_\Gamma= \frac{1}{\vert \Theta_j\vert^2}\big(-\alpha_j\partial_{\nu_\Gamma} P_j-\beta_j\partial_{\tau_\Gamma} P_j\big)\ .
$$
Whence \eqref{4} and \eqref{5} become
\bqnn
-V=\frac{1}{\vert \Theta_{\ii}\vert^2}\big(\alpha_{\ii}\partial_{\nu_\Gamma} P_{\ii}+\beta_{\ii}\partial_{\tau_\Gamma} P_{\ii}\big)=\frac{1}{\vert \Theta_{\oo}\vert^2}\big(\alpha_{\oo}\partial_{\nu_\Gamma} P_{\oo}+\beta_{\oo}\partial_{\tau_\Gamma} P_{\oo}\big)\qquad\text{on}\quad \Gamma(t)\, ,
\eqnn
{and},
\bqnn
\alpha_{\oo}\partial_{\nu} P_{\oo}+\beta_{\oo}\partial_{\tau} P_{\oo}=0\qquad\text{on}\quad [\vert x\vert=R]\, ,
\eqnn
{respectively,} where $\nu$ denotes the unit outer normal vector and $\tau$ the corresponding tangential 
vector on the cell boundary $[\vert x\vert=R]$. Therefore, we {arrive at} the following free boundary problem for the 
pressures~$P_j$
\begin{align}
 \Delta P_j&=0& &\text{in }\Om^j(t)\, ,& j=\text{i,o}\, ,\label{11a}\\
 P_\ii-P_\oo&=\sigma \kappa_{\Gamma(t)}+(\gamma_\oo-\gamma_\ii)|x|^2& &\text{on }\Gamma(t)\, , \label{11b}\\
 -V=\frac{1}{|\Theta_{\ii}|^2}\big(\alpha_{\ii}\partial_{\nu_\Gamma}P_{\ii}+\beta_{\ii}\partial_{\tau_\Gamma} P_{\ii}\big)
 &=\frac{1}{|\Theta_{\oo}|^2}\big(\alpha_{\oo}\partial_{\nu_\Gamma} P_{\oo}+\beta_{\oo}\partial_{\tau_\Gamma} P_{\oo}\big)& &\text{on }\:\Gamma(t)\, ,\label{11c}\\
 \alpha_{\oo}\partial_{\nu} P_{\oo}+\beta_{\oo}\partial_{\tau} P_{\oo}&=0& &\text{on }[|x|=R]\, ,\label{11d}
\end{align}
for $t>0$ {, complemented with an} initial surface $\Gamma(0)=\Gamma^0$.
Since only derivatives {and the difference of the two pressures} enter the system, {uniqueness can only be expected 
up to additive constants for them}. As we shall {later} see, this is one of the obstacles {that need to be overcome} 
from a mathematical viewpoint.

\subsection{Main Theorems}

To give a precise {formulation} of our mathematical results on \eqref{11a}-\eqref{11d}, we parametrize the boundary 
$\Gamma(t)$ over the unit sphere $\sph:=\{x\in\R^2\,;\,\vert x\vert=1\}$. {To this end} we introduce, 
for $s\ge 0$, {so-called} little Hölder spaces $h^{s}(U)$ over an open subset $U$ of $\R^n$ as the closure of 
$BUC^\infty(U)$ in $BUC^s(U)$. Here $BUC^s(U)$ {consists of} all functions $f:U\rightarrow\R$ {with} 
bounded and uniformly continuous derivatives up to order $[s]$ and with uniformly $(s-[s])$-Hölder continuous derivatives of order 
$[s]$. If $M$ is a (sufficiently smooth) submanifold of $\R^n$, we define $h^{s}(M)$ by means of an atlas for $M$ {
in the canonical way}.
In the following we shall identify a function $h:\sph\rightarrow\R$ with the function $\bar{h}:[0,2\pi]\rightarrow\R$ given by $\bar{h}(\theta):=h\big(e^{i\theta}\big)$ for $0\le\theta\le 2\pi$. {The bar} will often be dropped as no confusion seems likely.

Fix $a\in (0,1/4)$ and $\delta\in (0,1)$ {and set}
$$
\mathcal{V}:=\{\rho\in h^{4+\delta}(\sph)\,;\, \|\rho\|_\infty <a\}\, ,
$$
and, for $\rho\in\mathcal{V}$,
$$
\Omega_\rho^\ii:=\left\{x\in\R^2\setminus\{0\}\,;\,\vert x\vert <1+\rho\left(\frac{x}{\vert x\vert }\right)\right\}\cup \{0\}\, ,\quad \Om_\rho^\oo:=\Omega\setminus \overline{\Om_\rho^\ii}
$$
with $\Om:=\B(0,R)$ where, we recall, $R\geq 2$. Then
$$
 \Gamma_\rho:=\left\{x\in\R^2\,;\, \vert x\vert =1+\rho\left(\frac{x}{\vert x\vert }\right)\right\}=\Big\{
 {\big[1+\rho(y)\big]} y\,;\, y\in\sph\Big\}\ 
$$
separates the domains $\Omega_\rho^\ii$ and $\Omega_\rho^\oo$. Since $\Gamma_\rho$ can be described as the zero level set of 
$$
N_\rho(x):=\vert x\vert -1-\rho\left(\frac{x}{\vert x\vert }\right)\, ,\quad 3/4 <\vert x\vert <5/4\, ,
$$
with $N_\rho<0$ in $\Om_\rho^\ii\cap [3/4<\vert x\vert < 5/4]$, the unit normal $\nu_\rho(x)$ at $x\in \Gamma_\rho$ pointing 
from  $\Omega_\rho^\ii$ to $\Omega_\rho^\oo$ is {given by}
\bqn\label{12}
\nu_\rho(x)=\frac{\nabla N_\rho(x)}{\vert\nabla N_\rho(x)\vert}\ .
\eqn
In the following, we let $\tau_\rho=-(\vec{z}\times \nu_\rho)$ denote the corresponding tangential vector.
Next, suppose that $\rho\in C\big([0,T],\mathcal{V}\big)\cap C^1\big([0,T],h^{1+\delta}(\sph)\big)$ for some $T>0$ 
and set
$$
N_\rho(t,x):=N_{\rho(t)}(x)=\vert x\vert -1-\rho\left(t,\frac{x}{\vert x\vert }\right)\, ,\quad t\in [0,T]\, ,\quad\quad 3/4 <\vert x\vert <5/4\ .
$$
Observe that the normal velocity $V_{\rho}$ of the moving boundary $\Gamma_{\rho}$ equals 
$-\partial_t N_\rho /\vert \nabla N_\rho\vert$. Consequently, if $\rho\in C\big([0,T],\mathcal{V}\big)\cap C^1\big([0,T],h^{1+\delta}(\sph)\big)$  describes the evolution of the boundary, then \eqref{11a}-\eqref{11d} read as
\begin{align}
 \Delta P_j&=0& & \text{in }\Om_{\rho(t)}^j\, , &\ j=\text{i,o}\, ,\label{15a}\\
 P_\ii-P_\oo&=\sigma \kappa_{\Gamma_{\rho(t)}}+(\gamma_\oo-\gamma_\ii){\Big[ 1+\rho\bigl(t,\frac{x}{|x|}\bigr)\Big]}^2&  
 &\text{on }\Gamma_{\rho(t)}\, , &\label{15b}\\
\partial_t\rho\bigl(t,\frac{x}{|x|}\bigr)&=\frac{-1}{\vert \Theta_{j}\vert^2}\big[\alpha_{j}\nabla P_{j}+
\beta_{j}(\vec{z}\times\nabla P_{j})\big]\cdot\nabla N_{\rho(t)}& &\text{on }\Gamma_{\rho(t)}\, , &\ j=\text{i,o}\, ,\label{15c}\\
\alpha_{\oo}\partial_{\nu} P_{\oo}+\beta_{\oo}\partial_{\tau} P_{\oo}&=0&  &\text{on }\partial\Om=R\,\sph\, ,\label{15d}
\end{align}
for $t>0$ {and }with $\Gamma_{\rho(0)}=\Gamma^0$.
We call a triple $(\rho,P_\ii , P_\oo)$ a {\it (classical Hölder) solution to \eqref{15a}-\eqref{15d}} provided 
$$
\rho\in C\big([0,T],\mathcal{V}\big)\cap C^1\big([0,T],h^{1+\delta}(\sph)\big)
$$ 
for some $T>0$ and 
$$
P_j(t,\cdot)\in h^{2+\delta}\big(\Om_{\rho(t)}^j\big)\, ,\quad t\in [0,T]\, ,\quad j=\ii, \oo\, ,
$$ 
satisfy \eqref{15a}-\eqref{15d} pointwise. The main existence result is stated in the following theorem, a proof of which is given in Section~\ref{existence}:

\begin{thm}\label{T1}
Let $\delta\in (0,1)$. There is an open zero neighborhood $\mathcal{O}$ in $h^{4+\delta}(\sph)$ such that for each initial geometry $\rho_0\in\mathcal{O}$, there exist a time $T:=T(\rho_0)>0$ and a classical solution $(\rho,P_\mathrm{i} , P_\mathrm{o})$ to \eqref{15a}-\eqref{15d} with $\rho(0)=\rho_0$. This solution is unique except for additive constants in the pressures $P_\mathrm{i}$ and $P_\mathrm{o}$.
\end{thm}

We also investigate stability properties of the trivial solution corresponding to the geometry \mbox{$\rho= 0$}, that is, to the unit 
{circle} $\sph$, and constant pressures $P_\ii= c+\sigma+\gamma_\oo-\gamma_\ii$ and $P_\oo=c$ with arbitrary $c\in\R$. 
{The stability properties of the trivial solution are determined by the relative density of the fluids.}

\begin{thm}\label{T2}
If $\varrho_\mathrm{o}>\varrho_\mathrm{i}$ then the trivial solution $(\rho,P_\mathrm{i}, P_\mathrm{o})=\big(0,c+\sigma+\gamma_\mathrm{o}-\gamma_\mathrm{i},c\big)$ is locally asymptotically stable. However, if $\varrho_\mathrm{i}>\varrho_\mathrm{o}$, then the trivial solution $(\rho,P_\mathrm{i} , P_\mathrm{o})=\big(0,c+\sigma+\gamma_\mathrm{o}-\gamma_\mathrm{i},c\big)$ is unstable.
\end{thm}

We refer to Section~\ref{stable} for more precise statements of the stability {results} and their proof.

\section{Proof of Theorem~\ref{T1}}\label{existence}

{The proof of Theorem~\ref{T1} is best carried out in a coordinate system in which the moving interface between the 
liquids is fixed. We therefore begin with the transformation to a fixed domain formulation.}

\subsection{An Equivalent Problem on Fixed Domains}

We transform the free boundary problem to fixed domains using {the} standard Hanzawa-transform. 
Let \mbox{$\varphi:\R\rightarrow [0,1]$} be a smooth function with \mbox{$\varphi (r)=1$} for $\vert r\vert\le a$ 
and $\varphi (r)=0$ for $\vert r\vert\ge 3a$ and \mbox{$\|\varphi'\|_\infty <1/a$}. 
We introduce a family of $C^{4+\delta}$-diffeo\-mor\-phisms 
$$
\phi_\rho\in \mathrm{Diff}^{4+\delta}\big(\R^2,\R^2\big)\cap \mathrm{Diff}^{4+\delta}\big(B^\ii,\Om_\rho^\ii\big)\cap 
\mathrm{Diff}^{4+\delta}\big(B^\oo,\Om_\rho^\oo\big)\, , \quad \rho\in\mathcal{V}\, ,
$$
where 
$$
B^\ii:=\mathbb{B}(0,1)\, ,\quad B^\oo:=\Om\setminus\overline{B^\ii}\, ,
$$
by setting
$$
\phi_\rho(x):=\begin{cases} 
 \Bigl(|x|+[|x| -1]\rho\bigl(\frac{x}{|x|}\bigr)\Bigr)\frac{x}{|x|}&\text{if }0<|x|< 2\, ,\\ 
 x& \text{else.}
\end{cases}
$$
Note that $\phi_\rho$ maps  $\sph$ onto $\Gamma_\rho$. Given $\rho\in\mathcal{V}$, let
$$
\phi_\rho^*: BUC(\Om_\rho^j)\rightarrow BUC(B^j)\, ,\quad p\mapsto p\circ \phi_\rho
$$
denote the push-forward operator and
$$
\phi_*^\rho: BUC(B^j)\rightarrow  BUC(\Om_\rho^j)\, ,\quad q\mapsto q\circ \phi_\rho^{-1}
$$
the pull-back operator induced by $\phi_\rho$. 
{Given  $\rho\in\mathcal{V}$, the transformed differential and boundary operators acting on \mbox{$Q_j:=\phi_\rho^* P_j$} 
$j=\ii,\oo$ are given by}
$$
\mathcal{A}_j(\rho) Q_j:=\phi_\rho^*\big(\Delta (\phi_*^\rho Q_j)\big)
$$
and by
$$
\mathcal{B}_j(\rho) Q_j :=\frac{-1}{\vert\Theta_j\vert^2}\left[\alpha_j\big(\phi_\rho^*\nabla(\phi_*^\rho Q_j)\big)+
\beta_j\big(\vec{z}\times \big(\phi_\rho^*\nabla(\phi_*^\rho Q_j)\big)\big)\right]\cdot \big(\phi_\rho^*\nabla N_\rho\big)\, .
$$
{Defining}
$$
\mathcal{K}(\rho):=\sigma \phi_\rho^*\kappa_{\Gamma_\rho} +(\gamma_\oo-\gamma_\ii)\vert 1+\rho\vert^2
$$
for $\rho\in\mathcal{V}$ and
$$
\mathcal{B}_\oo :=\alpha_\oo\partial_\nu +\beta_\oo\partial_\tau \quad \text{on}\quad \partial\Om= R\sph\, ,
$$
the free boundary problem \eqref{15a}-\eqref{15d} is transformed to the following problem on fixed domains:
\begin{align}
 \mathcal{A}_j(\rho)Q_j&=0& &\text{in }B^j\, ,\quad  j=\text{i,o}\, ,\label{21a}\\
 Q_\ii-Q_\oo&=\mathcal{K}(\rho)&  &\text{on }\sph\, ,  \label{21b}\\
 \partial_t\rho&=\mathcal{B}_j(\rho)Q_j& &\text{on }\sph\, ,\quad j=\text{i,o}\, ,\label{21c}\\
 \mathcal{B}_\oo Q_\oo&=0& &\text{on } R\sph\ .&\label{21d}
\end{align}
{We} call a triple $(\rho,Q_\ii , Q_\oo)$ a {\it (classical Hölder) solution to \eqref{21a}-\eqref{21d}} provided 
$$
\rho\in C\big([0,T],\mathcal{V}\big)\cap C^1\big([0,T],h^{1+\delta}(\sph)\big)
$$ 
for some $T>0$ and 
$$
Q_j(t,\cdot)\in h^{2+\delta}\big(B^j\big)\, ,\quad t\in [0,T]\, ,\quad j=\ii, \oo\, ,
$$ 
satisfy \eqref{21a}-\eqref{21d} pointwise. With this definition, problems
\eqref{15a}-\eqref{15d} and \eqref{21a}-\eqref{21d} are equivalent {as follows from the next proposition.}

\begin{prop}\label{P7}
{Let $Q_j=\phi_\rho^*P_j$ and $P_j=\phi_*^\rho Q_j$. Then} $(\rho,P_\mathrm{i} , P_{\mathrm{o}})$ is a solution to 
\eqref{15a}-\eqref{15d}  if and only if $(\rho,Q_\mathrm{i} , Q_\mathrm{o})$ is a solution to \eqref{21a}-\eqref{21d}.
\end{prop}

\begin{proof}
Due to the above definitions of the differential operators, we merely {need to ascertain that no regularity loss 
occurs in the process}, i.e., that $P_j\in h^{2+\delta}(\Om_\rho^j)$ implies $Q_j=\phi_\rho^*P_j\in h^{2+\delta}(B^j)$ 
for $\rho\in\mathcal{V}$ and vice versa. This, however, easily follows from the fact that $\phi_\rho^*P_j\in BUC^4(B^j)$ 
when $P_j\in BUC^\infty(\Om_\rho^j)$ and \mbox{$\rho\in\mathcal{V}\subset h^{4+\delta}(\sph)$} and the observation 
that $h^{2+\delta}(B^j)$ coincides with the closure  of $BUC^4(B^j)$ in $BUC^{2+\delta}(B^j)$.
\end{proof}

{Next we} collect some regularity properties of the operators defined above.
\begin{lem}\label{L8}
(i) The operator $\mathcal{A}_j(\rho)\in \mathcal{L}\big(h^{2+\delta}(B^j),h^\delta (B^j)\big)$ is uniformly elliptic and 
analytic in $\rho\in\mathcal{V}$ {for $j=\mathrm{i}, \mathrm{o}$}.\\
(ii) The operator $\mathcal{B}_j(\rho)\in \mathcal{L}\big(h^{2+\delta}(B^j),h^\delta (\sph)\big)$ is analytic in 
$\rho\in\mathcal{V}$ {for $j=\mathrm{i}, \mathrm{o}$}.
\end{lem}

\begin{proof}
(i) Given $\rho\in\mathcal{V}$, uniform ellipticity of $\mathcal{A}_j(\rho)$ is a consequence of its symbol being 
$$
a(\rho)(\xi)=\phi_\rho^*\big(\vert J(\phi_\rho^{-1})\xi\vert^2+\Delta \phi_\rho^{-1}\cdot\xi\big)\, ,\quad \xi\in\R^2\, ,
$$
where the Jacobian $J(\phi_\rho^{-1})$ of $\phi_\rho^{-1}$ satisfies 
$\vert J(\phi_\rho^{-1}) \xi\vert^2\ge\epsilon\vert\xi\vert^2$ for some $\epsilon>0$. Analyticity in $\rho$ follows from 
the analyticity of $\phi_\rho$ and of $J(\phi_\rho^{-1})=\phi_\rho^* J(\phi_\rho)$. For details 
we refer to \cite[$\S$ 3.2]{EscherEhrnstroemMatioc}.

(ii) Note that
\bqn\label{22}
\nabla\rho\left(\frac{x}{\vert x\vert}\right)=\rho'\left(\frac{x}{\vert x\vert}\right)\frac{1}{\vert x\vert^2}\left(
\begin{matrix} -x_2\\  x_1\end{matrix}\right)\, ,\quad x=(x_1,x_2)\in\R^2\setminus\{0\}\, ,\quad \rho\in C^1(\sph)\ .
\eqn
Hence, for $\rho\in\mathcal{V}$ and $\phi_\rho=(\phi_\rho^1,\phi_\rho^2)$,
$$
\big(\phi_\rho^*\nabla N_\rho\big)(y)=\frac{\phi_\rho(y)}{\left\vert \phi_\rho(y)\right\vert}-\rho' \left(\frac{\phi_\rho(y)}{\vert \phi_\rho(y)\vert}\right) \frac{1}{\left\vert\phi_\rho(y)\right\vert^2} \left(\begin{matrix} -\phi_\rho^2(y)\\  \phi_\rho^1(y)\end{matrix}\right)\, ,\quad y\in\sph\, ,
$$
showing that {$\bigl[\rho\mapsto \phi_\rho^*\nabla N_\rho\bigr]\, ,\: \mathcal{V}\rightarrow h^{3+\delta}(\sph)$ is analytic}. 
{The definition of $\mathcal{B}_j(\rho)$ entails its analytic dependence on $\rho$ since}
$$
\phi_\rho^*\nabla (\phi_*^\rho Q_j)=\nabla Q_j\, \phi_\rho^*J(\phi_\rho^{-1})\ .
$$
\end{proof}

The curvature operator $\mathcal{K}$ {can easily be computed explicitly.}

\begin{lem}\label{L10}
The operator $\mathcal{K}:\mathcal{V}\rightarrow h^{2+\delta}(\sph)$ is analytic and given by
$$
 \mathcal{K}(\rho)=\sigma\frac{(1+\rho)^2+2\dot{\rho}^2-(1+\rho)\ddot{\rho}}{\big[(1+\rho)^2+\dot{\rho}^2\big]^{3/2}} 
 +(\gamma_\mathrm{o}-\gamma_\mathrm{i})(1+\rho)^2\, ,\quad\rho\in\mathcal{V}\, ,
$$
and 
$$
\partial\mathcal{K}(0)[h]=\sigma(-h-\ddot{h})+2 (\gamma_\mathrm{o}-\gamma_\mathrm{i})h\, ,\quad h\in h^{4+\delta}(\sph)\ .
$$
\end{lem}

\begin{proof}
If $\rho\in \mathcal{V}$, a parametrization of $\Gamma_\rho$ is given by
$$
(1+\rho(\theta))e^{i\theta}=:x(\theta)+ iy(\theta)\, ,\quad \theta\in [0,2\pi]\ .
$$
Hence the claim is a consequence of {the curvature formula}
$$
\kappa_{\Gamma_\rho}=\frac{\dot{x}\ddot{y}-\ddot{x}\dot{y}}{\big(\dot{x}^2+\dot{y}^2\big)^{3/2}}\, , 
$$
and the definition of $\mathcal{K}(\rho)$.
\end{proof}

\subsection{Local Well-Posedness}

In order to establish local existence and uniqueness of solutions, we split \eqref{21a}-\eqref{21d} in three subproblems: 
two elliptic problems for the transformed pressures and an evolution equation for the geometry. More precisely, 
given $\rho\in\mathcal{V}$ and a function $h$ defined on $\sph$, we shall first look for solutions $Q_\ii=S(\rho,h)$ 
to the {following} elliptic problem {on $B^\ii$} with Neumann type boundary condition
\begin{align}
 \mathcal{A}_\mathrm{i}(\rho)Q_\ii\, & =0&  &\text{in }B^\ii\, ,\label{P1a}\\
 \mathcal{B}_\ii(\rho)Q_\ii &=h- \frac{\big\vert \phi_\rho^*(\nabla N_\rho)\big\vert}{\vert\Gamma_\rho\vert}\int_{\Gamma_\rho} 
 \frac{\phi_*^\rho h}{\vert \nabla N_\rho\vert}\, \rd\sigma & &\text{on }\sph\, , \label{P1b}\\
 \int_{\sph}Q_\ii\,\rd \sigma & =\int_{\sph} h\,\rd \sigma\, ,&\label{P1c}
\end{align}
where $\vert\Gamma_\rho\vert$ stands for the length of the closed curve $\Gamma_\rho$.
{In a second step}, we shall study the elliptic problem {on the annulus $ B^\oo$}  with Dirichlet-Neumann boundary 
conditions
\begin{align}
\mathcal{A}_\mathrm{o}(\rho)Q_\oo =0  &\qquad\text{in  }B^\oo\, ,\label{P2a}\\
 Q_\oo= g&\qquad\text{on  }\sph\, , \label{P2b}\\
 \mathcal{B}_\oo Q_\oo=0&\qquad\text{on  }R\, \sph\, ,\label{P2c}
\end{align}
with solution $Q_\oo=T(\rho,g)$ depending on $\rho\in\mathcal{V}$ and $g$ on $\sph$. Finally {we use the solution operators 
$S$ and $T$ to derive the evolution equation}
\bqn\label{P3}
\partial_t\rho=\mathcal{B}_\oo(\rho) T\big(\rho, S(\rho,\partial_t\rho)-\mathcal{K}(\rho)\big)\ .
\eqn
{for $\rho$}. Note that $\partial_t\rho$ {appears} on both sides of \eqref{P3}. As \eqref{P1a}-\eqref{P1b} is a 
Neumann-type problem, an additional integral term is introduced on the right hand of side \eqref{P1b} {which makes the 
mean zero} and {thus} ensures solvability of the problem. The third equation \eqref{P1c} guarantees uniqueness. 
Let us point out that, provided we can solve \eqref{P3} for $\rho=\rho(t)$, the particular form of the boundary operator 
$\mathcal{B}_\oo(\rho)$ and the properties of $T(\rho,g)$ {yield} (see \eqref{just} below)
$$
\int_{\Gamma_\rho} \frac{\phi_*^\rho \partial_t\rho }{\vert \nabla N_\rho\vert}\, \rd\sigma =0\ .
$$
Thus the integral term in \eqref{P1b} vanishes for $h=\partial_t\rho$ and we may take 
$$
Q_\ii(t)=S\big(\rho(t),\partial_t\rho(t)\big)\, ,\qquad Q_\oo(t)=T\big(\rho(t),Q_\ii(t)-\mathcal{K}(\rho(t))\big)
$$ 
to obtain a solution for the original equations \eqref{21a}-\eqref{21d}. {We shall be more specific at the end of this section.}

In the following two propositions we study the solvability of the elliptic problems.

\begin{prop}\label{P12}
Given $\rho\in\mathcal{V}$ and $h\in h^{1+\delta}(\sph)$, {\eqref{P1a}-\eqref{P1c} possesses} a unique solution 
$$
 Q_\mathrm{i}=S(\rho,h)\in h^{2+\delta}(B^\mathrm{i})\, ,
$$ 
{and the map} 
$$
 [\rho\mapsto S(\rho,\cdot)]:\mathcal{V}\rightarrow\mathcal{L}\big(h^{1+\delta}(\sph),h^{2+\delta}(B^\mathrm{i})\big)
$$
is analytic.
\end{prop}

To prove this statement {on the fixed domain $B^\ii$ we shall first prove an auxiliary result on the domain~$\Omega_\rho^\ii$:}

\begin{lem}\label{LA}
Let $\rho\in\mathcal{V}$ and let $\nu=\nu_{\Gamma_\rho}$ and $\tau=\tau_{\Gamma_\rho}$ denote the corresponding outer 
unit normal {and unit} tangential vector{s to} $\Gamma_\rho$, respectively. Set 
$$
 \mu:=\frac{-1}{\vert\Theta_\mathrm{i}\vert^2}\big(\alpha_\mathrm{i}\nu+\beta_\mathrm{i}\tau\big)
$$
and define 
$$
 \mathbb{A}:h^{2+\delta}(\Om_\rho^\mathrm{i})\rightarrow h^\delta(\Om_\rho^\mathrm{i})\times h^{1+\delta}(\Gamma_\rho)\, ,
 \quad u\mapsto \big(-\Delta u,\partial_\mu u\big)\ .
$$
Then $\,\mathrm{ker}(\mathbb{A})=\R\cdot{\bf 1}$ and 
$$
 {\mathrm{im}}(\mathbb{A})=\left\{(f,g)\in h^\delta(\Om_\rho^\mathrm{i})\times h^{1+\delta}(\Gamma_\rho)\, ;\, 
 \int_{\Om_\rho^\mathrm{i}} f\,\rd x = \frac{\vert\Theta_\mathrm{i}\vert^2}{\alpha_\mathrm{i}}\int_{\Gamma_\rho}g\,
 \rd \sigma\right\}\ .
$$
\end{lem}

\begin{proof}
Note that $\mu$ is nowhere tangential and that ${\bf 1}$ is an eigenfunction {of} $\mathbb{A}$. 
Thus, \mbox{$\mathrm{ker}(\mathbb{A})=\R\cdot{\bf 1}$} {follows from} \cite[Thm.12.1]{AmannIsrael} since 
$(-\Delta ,\partial_\mu)$ is regular elliptic. To determine the range of $\mathbb{A}$, fix $p>n/(1-\delta)$ 
and suppose {first that} $\mathbb{A}u=(f,g)$, that is, 
\bqn\label{X1}
-\Delta u=f\quad \text{in}\ \Om_\rho^\ii\, ,\qquad \partial_\mu u=g\quad \text{on}\ \Gamma_\rho\ .
\eqn
Since $\partial_\tau w=(\vec{z}\times \nabla w)\cdot\nu$ and $\mathrm{div}(\vec{z}\times \nabla w)=0$, we have by Gauss' Theorem
\bqn\label{X3}
\int_{\Gamma_\rho}\partial_\tau w\,\rd \sigma =\int_{\Om_\rho^\ii} \mathrm{div}(\vec{z}\times \nabla w)\,\rd x =0\, ,\quad w\in W_p^2(\Om_\rho^\ii)\ .
\eqn
Therefore,
\bqnn
\int_{\Om_\rho^\mathrm{i}} f\,\rd x =
\frac{\vert\Theta_\mathrm{i}\vert^2}{\alpha_\mathrm{i}}\int_{\Gamma_\rho}\left(
-\frac{\alpha_\mathrm{i}}{\vert\Theta_\mathrm{i}\vert^2}\partial_\nu u-\frac{\beta_\mathrm{i}}{\vert\Theta_\mathrm{i}\vert^2}\partial_\tau u\right)
\,\rd \sigma
=\frac{\vert\Theta_\mathrm{i}\vert^2}{\alpha_\mathrm{i}}\int_{\Gamma_\rho}g\,\rd \sigma\ .
\eqnn
For the reverse inclusion we use a Fredholm argument. By \cite[Lem.5.1]{AmannIsrael} there is a coretraction $$M\in\mathcal{L}\big(W_p^{1-1/p}(\Gamma_\rho),W_p^2(\Om_\rho^\ii)\big)$$ for the boundary operator such that $\partial_\mu Mg=g$ for \mbox{$g\in  W_p^{1-1/p}(\Gamma_\rho)$}. Then $$R:=\Delta M\in \mathcal{L}\big(W_p^{1-1/p}(\Gamma_\rho),L_p(\Om_\rho^\ii)\big)\, ,$$ and finding a solution \mbox{$u\in W_p^2(\Om_\rho^\ii)$} to problem \eqref{X1} for a given \mbox{$(f,g)\in L_p(\Om_\rho^\ii)\times W_p^{1-1/p}(\Gamma_\rho)$} is equivalent to finding a solution $v\in W_p^2(\Om_\rho^\ii)$ to 
\bqn\label{X1a}
-\Delta v=f+Rg\quad \text{in}\ \Om_\rho^\ii\, ,\qquad \partial_\mu v=0\quad \text{on}\ \Gamma_\rho\, ,
\eqn
{and} setting $u:=v+Mg$. Let 
$$
W_{p,\mu}^2\big(\Om_\rho^\ii\big):=\left\{v\in W_p^2(\Om_\rho^\ii)\,;\, \partial_\mu v =0\ \text{on}\ \Gamma_\rho\right\}
$$ 
and consider the closed linear operator $T$ on $L_p(\Om_\rho^\ii)$ given by $Tv:=-\Delta v$ for $v\in W_{p,\mu}^2(\Om_\rho^\ii)$. Then, as above, $\mathrm{ker}(T)=\R\cdot {\bf 1}$, and $T$ has compact resolvent. Hence $T$ is a Fredholm operator on $L_p(\Om_\rho^\ii)$ with index zero. Since obviously $W_{p,\mu}^2(\Om_\rho^\ii)$ is dense in $L_p(\Om_\rho^\ii)$, we also obtain for its dual $T'$ that  $\mathrm{ker}(T')=\R\cdot  f'$ for some $f'\in L_{p'}(\Om_\rho^\ii)$, where $p'$ denotes the dual exponent of~$p$. Moreover, $Tv=h$ with $h\in L_p(\Om_\rho^\ii)$ is solvable for $v\in W_{p,\mu}^2(\Om_\rho^\ii)$ if and only if $\langle f',h\rangle=0$. However, since \eqref{X3} ensures
$$
\int_{\Om_\rho^\ii}\Delta v\ \rd x = -\frac{\vert\Theta_\ii\vert^2}{\alpha_\ii}\int_{\Gamma_\rho}\partial_\mu v\ \rd\sigma=0\, ,\quad v\in W_{p,\mu}^2(\Om_\rho^\ii)\, ,
$$
we may take $f'={\bf 1}$, that is, $\mathrm{ker}(T')=\R\cdot {\bf 1}$, and $Tv=f+Rg$ is solvable for $v\in W_{p,\mu}^2(\Om_\rho^\ii)$ with $$(f,g)\in L_p(\Om_\rho^\ii)\times W_p^{1-1/p}(\Gamma_\rho)$$ if and only if
\bqn\label{rg}
0=\int_{\Om_\rho^\ii} (f+Rg)\ \rd x= \int_{\Om_\rho^\ii} f\ \rd x - \frac{\vert\Theta_\ii\vert^2}{\alpha_\ii}\int_{\Gamma_\rho}\partial_\mu g\ \rd\sigma \, ,
\eqn
the last equality being again due to \eqref{X3}. Finally, if $u$ is the solution to \eqref{X1} for a given 
$$(f,g)\in h^\delta(\Om_\rho^\ii)\times h^{1+\delta}(\Gamma_\rho){\subset} 
L_p(\Omega_\rho^\ii)\times W_p^{1-1/p}(\Gamma_\rho)$$ satisfying \eqref{rg}, 
then $u\in W_p^2(\Om_\rho^\ii)\hookrightarrow C^{1+\delta}(\overline{\Om_\rho^\ii})$ 
by the choice of $p$. Using Schauder regularity theory 
{\cite[Thm.6.15]{GilbargTrudinger}} for
$$
-\Delta \bar{u}=f\quad \text{in}\ \Om_\rho^\ii\, ,\qquad \bar{u}+\partial_\mu \bar{u}=g+u\quad \text{on}\ \Gamma_\rho\, ,
$$
then easily gives $u=\bar{u}\in h^{2+\delta}(\Om_\rho^\ii)$. This proves the claimed 
characterization of ${\mathrm{im}}(\mathbb{A})$.
\end{proof}

\begin{proof}[Proof of Proposition~\ref{P12}]
For $\rho\in\mathcal{V}$ fixed, set
$$
P(\rho) h:=\frac{\vert \phi_\rho^*(\nabla N_\rho)\vert}{\vert\Gamma_\rho\vert}
\int_{\Gamma_\rho} \frac{\phi_*^\rho h}{\vert \nabla N_\rho\vert}\, \rd\sigma\, ,
\quad h\in h^{1+\delta}(\sph)\, ,
$$
and define an operator $\mathbb{A}_\ii(\rho): h^{2+\delta}(B^\ii)\rightarrow 
h^\delta(B^\ii)\times h^{1+\delta}(\sph)$ by
$$
\mathbb{A}_\ii(\rho) u:=\big(D(\rho) \mathcal{A}_\ii(\rho)u\,,\, d(\rho)
\mathcal{B}_\ii(\rho) u\big), \quad u\in h^{2+\delta}(B^\ii)\, ,
$$
where $D(\rho):=\vert\mathrm{det}\,\partial_x \phi_\rho\vert$ and $d(\rho):=1+\rho$. 
Then, since $D(\rho)>0$ and $d(\rho)>0$, we readily obtain from Lemma~\ref{LA} 
that $\mathrm{ker}\big(\mathbb{A}_\ii(\rho)\big)=\R\cdot{\bf 1}$. Moreover, 
$(f,g)\in\mathrm{im}\big(\mathbb{A}_\ii(\rho)\big)$ if and only if 
$(\hat{f},\hat{g})\in\mathrm{im}(\mathbb{A})$, where 
$$
\hat{f}:=\phi_*^\rho \frac{f}{D(\rho)}\in h^{\delta}(\Om_\rho^\ii)\, ,\qquad \hat{g}:=\frac{1}{\vert \nabla N_\rho \vert} \phi_*^\rho \frac{g}{d(\rho)}\in h^{1+\delta}(\Gamma_\rho)\ .
$$
{By virtue of \eqref{22} we have that}
\bqn\label{n}
 \vert \nabla N_\rho \vert =\frac{\sqrt{\dot{\rho}+(1+\rho)^2}}{1+\rho}\quad
 \text{on}\quad\sph\, ,
\eqn
{and we infer from Lemma~\ref{LA} that}
$$
\mathrm{im}\big(\mathbb{A}_\ii(\rho)\big)=\left\{(f,g)\in h^\delta(B^\ii)
\times h^{1+\delta}(\sph)\, ;\, \int_{B^\ii} f\,\rd x
=\frac{\vert\Theta_\mathrm{i}\vert^2}{\alpha_\mathrm{i}}\int_{\sph}g\,\rd 
\sigma\right\}=:Y
$$
{which is} independent of $\rho$. Consequently, the map
$$
\mathbb{A}_\ii(\rho): h^{2+\delta}(B^\ii)\big/_{\R\cdot{\bf 1}}\longrightarrow Y
$$
is an isomorphism. Furthermore, 
$
v=\mathcal{N}(\rho)(f,g):= \mathbb{A}_\ii(\rho)^{-1}(f,g)
$
is the unique solution in $h^{2+\delta}(B^\ii)\big/_{\R\cdot{\bf 1}}$ to
$$
D(\rho)\mathcal{A}_\ii(\rho)v=f\quad\text{on}\ B^\ii\, ,\qquad d(\rho)\mathcal{B}_\ii 
(\rho) v=g\quad\text{on}\ \sph
$$
{for each $\rho\in\mathcal{V}$ and $(f,g)\in Y$}.
As $D(\rho)$ and $d(\rho)$ are analytic in $\rho$, we deduce that 
$$
\mathcal{N}:\mathcal{V}\rightarrow \mathcal{L}\left(Y,h^{2+\delta}(B^\ii)\big/_{\R\cdot{\bf 1}}\right)
$$
is analytic. Now, we easily check with the help of \eqref{n} that
$$
\int_{\sph} d(\rho)\big(1-P(\rho)\big)h\, \rd \sigma=0\, ,
$$
whence $\big(0,d(\rho)(1-P(\rho))h\big)\in Y$ for $h\in h^{1+\delta}(\sph)$. Thus $u:=\mathcal{N}(\rho)\big(0,d(\rho)(1-P(\rho))h\big)$ is the unique solution in $h^{2+\delta}(B^\ii)\big/_{\R\cdot{\bf 1}}$ to
$$
\mathcal{A}_\ii(\rho)u=0\quad\text{on}\ B^\ii\, ,\qquad \mathcal{B}_\ii (\rho) u=h-P(\rho)h\quad\text{on}\ \sph\ .
$$
By \eqref{n} we have
$$
P(\rho)h=\frac{\sqrt{\dot{\rho}+(1+\rho)^2}}{(1+\rho)\int_0^{2\pi} \sqrt{\dot{\rho}+(1+\rho)^2}\,\rd\theta} \int_0^{2\pi} h(\theta)\, \big(1+\rho(\theta)\big)\,\rd \theta\, ,
$$
hence $P(\rho)$ depends analytically on $\rho$ and so does $u=u(\rho)$. Given $\rho\in\mathcal{V}$ and \mbox{$h\in h^{1+\delta}(\sph)$}, problem~\eqref{P1a}-\eqref{P1c} thus admits a unique solution $Q_\ii=S(\rho,h)$ and 
$$
[\rho\mapsto S(\rho,\cdot)]:\mathcal{V}\rightarrow \mathcal{L}\big(h^{1+\delta}(\sph),h^{2+\delta}(B^\mathrm{i})\big)
$$ 
is analytic.
This proves Proposition~\ref{P12}.
\end{proof}

\begin{prop}\label{P13}
Given $\rho\in\mathcal{V}$ and $g\in h^{2+\delta}(\sph)$, there is a unique solution $Q_\mathrm{o}=T(\rho,g)\in h^{2+\delta}(B^\mathrm{o})$ to \eqref{P2a}-\eqref{P2b}. Moreover, the map $[\rho\mapsto T(\rho,\cdot)]:\mathcal{V}\rightarrow \mathcal{L}\big(h^{2+\delta}(\sph),h^{2+\delta}(B^\mathrm{o})\big)$ is analytic.
\end{prop}

\begin{proof}
Let $\rho\in\mathcal{V}$ and define an operator $\mathbb{A}_\oo(\rho): h^{2+\delta}(B^\oo)\rightarrow h^{\delta}(B^\oo)\times h^{2+\delta}(\sph)\times h^{1+\delta}(R\,\sph)$ by
$$
\mathbb{A}_\oo(\rho) u:=\big(\mathcal{A}_\oo(\rho)u\,,\, {u\vert_{\mathbb{S}^1}}\, ,\,\mathcal{B}_\oo u\big), \quad u\in h^{2+\delta}(B^\oo)\ .
$$
Then $\mathbb{A}_\oo(\rho)$ is invertible. Indeed, to check its injectivity, let $\mathbb{A}_\oo(\rho)u=0$. Then, for $\bar{u}:=\phi_*^\rho u$,
\begin{align*}
\Delta \bar{u} =0  &\qquad \text{in}\quad \Om_\rho^\oo\, ,\\
\bar{u} = 0 &\qquad\text{on}\quad \Gamma_\rho\, , \\
\alpha_\oo\partial_\nu \bar{u}+\beta_\oo\partial_\tau \bar{u} = 0&\qquad
\text{on}\quad R\,\sph\, ,
\end{align*}
{and Green's formula yields}
\bqnn
\begin{split}
0&=\int_{\Om_\rho^\oo} \bar{u} \Delta \bar{u}\,\rd x=-\int_{\Om_\rho^\oo}\vert \nabla \bar{u}\vert^2\,\rd x +\int_{\Gamma_\rho \cup R\sph} \bar{u} \partial_\nu \bar{u}\,\rd\sigma
= -\int_{\Om_\rho^\oo}\vert \nabla \bar{u}\vert^2\,\rd x -\frac{\beta_\oo}{\alpha_\oo}\int_{\Gamma_\rho \cup R\sph} \frac{1}{2}\partial_\tau \bar{u}^2 \,\rd\sigma\\
&= -\int_{\Om_\rho^\oo}\vert \nabla \bar{u}\vert^2\,\rd x -\frac{\beta_\oo}{2\alpha_\oo}\int_{\Om_\rho^\oo} \mathrm{div}\left(\begin{matrix}-\partial_2 {\bar{u}^2}\\ \partial_1 {\bar{u}^2}\end{matrix}\right)\,\rd x = -\int_{\Om_\rho^\oo}\vert \nabla \bar{u}\vert^2\,\rd x\ .
\end{split}
\eqnn
This entails $\bar{u}=0$, i.e. $u=0$. To determine the range of $\mathbb{A}_\oo(\rho)$ consider again the transformed problem
\begin{align}
\Delta \bar{u} =\bar{f}  &\qquad\text{in}\quad \Om_\rho^\oo\, ,\label{l}\\
\bar{u} = \bar{g}&\qquad
\text{on}\quad \Gamma_\rho\, , \label{ll}\\
\alpha_\oo\partial_\nu \bar{u}+\beta_\oo\partial_\tau \bar{u} = \bar{h} &\qquad \text{on}\quad R\,\sph\ . \label{lll}
\end{align}
By means of a coretraction as in the proof of Lemma~\ref{LA} we may assume $\bar{g}=0$ and $\bar{h}=0$. However, by \cite{AmannIsrael}, 
$$
\lambda-\Delta: \big\{v\in W_2^2(\Om_\rho^\oo)\,;\, v=0\ \text{on}\ \Gamma_\rho\, ,
\, (\alpha_\oo\partial_\nu +\beta_\oo\partial_\tau ) v=0\ \text{on}\ R\,\sph\big\}
\longrightarrow L_2(\Om_\rho^\oo)
$$
is invertible for $\lambda\ge 0$ sufficiently large, and the argument above to 
prove injectivity implies that we actually may take $\lambda=0$.
Thus, \eqref{l}-\eqref{lll} is uniquely solvable for each 
$\bar{f}\in L_2(\Om_\rho^\oo)$, $\bar{g}\in H^{3/2}(\Gamma_\rho)$, and 
$\bar{h}\in H^{1/2}(R\,\sph)$. Schauder regularity theory ensures 
$\bar{u}\in h^{2+\delta}(\Om_\rho^\oo)$ provided $\bar{f}\in h^\delta(\Om_\rho^\oo)$, 
$\bar{g}\in h^{2+\delta}(\Gamma_\rho)$, and $\bar{h}\in h^{1+\delta}(R\,\sph)$. 
Consequently, $\mathbb{A}_\oo(\rho)$ is surjective and thus invertible, and the 
claim on the analyticity of $T(\rho,\cdot)=\mathbb{A}_\oo(\rho)^{-1}(0,\cdot,0)$ 
follows from Lemma~\ref{L8}.
\end{proof}

Next, we state a multiplier result that we shall use {later}. The proof is 
a straightforward modification of the case $r=s$ in \cite{ArendtBu}.

\begin{prop}\label{AB} 
Let $r,s\in (0,\infty)\setminus\N$ and $(M_n)_{n\in\Z}$ be a sequence in $\C$ satisfying
\begin{itemize}
\item[(i)] $\sup_{n\in\Z\setminus\{0\}}\vert n\vert^{r-s}\vert M_n\vert <\infty\, ,$
\item[(ii)] $\sup_{n\in\Z\setminus\{0\}}\vert n\vert^{r-s+1}\vert M_{n+1}-M_n\vert <\infty\, ,$
\item[(iii)] $\sup_{n\in\Z\setminus\{0\}}\vert n\vert^{r-s+2}\vert M_{n+2}-2M_{n+1}+M_n\vert <\infty\, .$
\end{itemize}
Then the mapping
$$
\sum_{n\in\Z} \hat{h}_n e^{in\theta}\longmapsto \sum_{n\in\Z}  \hat{h}_n M_n  e^{in\theta}
$$
belongs to $\mathcal{L}\big(C^s(\sph),C^r(\sph)\big)$.
\end{prop}

We now focus on the evolution equation \eqref{P3}, which we may rewrite, using 
linearity of the solution operators $S$ and $T$, as
\bqn\label{P3a}
\big(1-\mathcal{R}(\rho)\big)\partial_t\rho=\mathcal{B}_\oo(\rho) T\big(\rho,-\mathcal{K}(\rho)\big)
\eqn
with
$$
\mathcal{R}(\rho)z:= \mathcal{B}_\oo(\rho) T\big(\rho,S(\rho,z)\big)\, ,\quad z\in h^{1+\delta}(\sph)\, ,\quad \rho\in\mathcal{V}\ .
$$
To solve \eqref{P3a} for $\partial_t\rho$, we use the following

\begin{lem}\label{L14}
There is an open zero neighborhood $\mathcal{W}\subset\mathcal{V}$ in $h^{4+\delta}(\sph)$ such that \mbox{$1-\mathcal{R}(\rho)$} is an isomorphism on $h^{1+\delta}(\sph)$ for each $\rho\in\mathcal{W}$.
\end{lem}

\begin{proof}
By the smooth dependence on $\rho$ stated in Lemma~\ref{L10}, Proposition~\ref{P12}, and Proposition~\ref{P13}, it suffices to prove that \mbox{$1-\mathcal{R}(0)\in \mathcal{L}\big(h^{1+\delta}(\sph)\big)$} is invertible. {To do this} we {derive} 
its Fourier expansion in polar coordinates $(r,\theta)$. To compute $Q_\ii=S(0,h)$ we 
first note that, for $\rho=0$, problem \eqref{P1a}-\eqref{P1c} {becomes}
\begin{align*}
\left(\frac{1}{r}\partial_r(r\partial_r)+\frac{1}{r^2}\partial_\theta^2\right) Q_\ii&=0  &\text{in}&\quad [r<1]\, ,\\
\frac{-1}{\vert\Theta_\ii\vert} \left(\alpha_\ii\partial_r  -\beta_\ii\partial_\theta\right)  Q_\ii&=h- \frac{1}{2\pi}\int_0^{2\pi} h(\theta)\, \rd\theta
&\text{on}&\quad [r=1]\, , \\
\int_{0}^{2\pi}Q_\ii(1,\theta)\,\rd \theta &=\int_{0}^{2\pi} h(\theta)\,\rd \theta\, ,
\end{align*}
which, for a given $h\in h^{1+\delta}(\sph)$ with expansion
\bqn\label{h}
h(\theta)=\sum_{n\in\Z} \hat{h}_n e^{in\theta}\, ,
\eqn
has the unique solution 
\bqn\label{Qi0}
Q_\ii(r,\theta)= \hat{h}_0+\sum_{n\in\Z\setminus\{0\}} \hat{h}_n \frac{\vert\Theta_\ii\vert^2}{-\vert n\vert \alpha_\ii+in\beta_\ii}  r^{\vert n\vert} e^{in\theta}\ .
\eqn
Similarly, for $\rho=0$, problem \eqref{P2a}-\eqref{P2b} reads as
\begin{align*}
\left(\frac{1}{r}\partial_r(r\partial_r)+\frac{1}{r^2}\partial_\theta^2\right) Q_\oo&=0  &\text{in}&\quad [1<r<R]\, ,\\
Q_\oo&=g
&\text{on}&\quad [r=1]\, ,\\
\left(\alpha_\oo\partial_r -\frac{\beta_\oo}{R}\partial_\theta\right) Q_\oo&=0\, &\text{on}&\quad [r=R]\, ,
\end{align*}
and for $g\in h^{2+\delta}(\sph)$ with expansion
$$
g(\theta)=\sum_{n\in\Z} \hat{g}_n e^{in\theta}\, ,
$$
the unique solution $Q_\oo=T(0,g)$ is given by
\bqn\label{Qo0}
Q_\oo(r,\theta)= \hat{g}_0+\sum_{n\in\Z\setminus\{0\}} \hat{g}_n \left(\frac{\Theta_\oo}{\Theta_\oo+R^{2n}\bar{\Theta}_\oo}r^n+\frac{\bar{\Theta}_\oo}{\bar{\Theta}_\oo+R^{-2n}\Theta_\oo}r^{-n}\right) e^{in\theta}\ .
\eqn
Therefore, given $h\in h^{1+\delta}(\sph)$ with expansion \eqref{h}, we have
$$
\mathcal{R}(0)h=\frac{-1}{\vert\Theta_\oo\vert}\left(\alpha_\oo\partial_r-\beta_\oo\partial_\theta\right) T\big(0,S(0,h)\big)=\sum_{n\in\Z\setminus\{0\}} \hat{h}_n\, l_n e^{in\theta}\, ,
$$
where
\bqn\label{36A}
l_n:=\frac{\big(1-R^{2n}\big)\vert\Theta_\ii\vert^2}{\big(\mathrm{sign}(n)\alpha_\ii-i\beta_\ii\big)\big(\Theta_\oo+R^{2n}\bar{\Theta}_\oo\big)}\, ,\quad n\in\Z\setminus\{0\} \ .
\eqn
We next {use} Proposition~\ref{AB} with $M_n:=(1-l_n)^{-1}$ to check that 
\mbox{$1-\mathcal{R}(0)$} is invertible on $C^{1+\delta}(\sph)$. Note that $l_n\not= 1$ 
for each $n\in\Z\setminus\{0\}$ since $\alpha_j >0$, $j=\ii , \oo$. Also,
\bqn\label{xx}
\lim_{n\rightarrow\infty}l_n=-\frac{\Theta_\ii \Theta_\oo}{\vert \Theta_\oo\vert^2}\not= 1\, ,\qquad
\lim_{n\rightarrow -\infty}l_n=-\frac{\bar{\Theta}_\ii \bar{\Theta}_\oo}{\vert \Theta_\oo\vert^2}\not= 1
\eqn
so that 
\bqn\label{39}
\sup_{n\in\Z\setminus\{0\}} \vert M_n\vert <\infty\ .
\eqn
Next,
$$
M_{n+1}-M_n=\frac{l_{n+1}-l_n}{(1-l_{n+1})(1-l_n)}\, ,
$$
where, for $n\ge 1$,
$$
l_{n+1}-l_n=\frac{O(R^{2n+2})}{\bar{\Theta}_\oo^2 R^{4n+2}+O(R^{2n+2})}\, ,\quad n\ge 1\, ,\qquad
l_{n+1}-l_n=\frac{O(R^{2n})}{\bar{\Theta}_\oo^2 +O(R^{2n})}\, ,\quad n\le -1\ .
$$
Since $nR^{-2n}\rightarrow 0$ as $n\rightarrow \infty$ and $n R^{2n}\rightarrow 0$ as $n\rightarrow -\infty$, it follows from \eqref{39} that
\bqnn
\sup_{n\in\Z\setminus\{0\}} \vert n\vert \, \vert M_{n+1}- M_n\vert <\infty\ .
\eqnn
Finally, as above, we obtain from  \eqref{xx} and \eqref{39} {that}
\bqnn
\begin{split}
n^2\,\big\vert M_{n+2}-2M_{n+1}&+M_n\big\vert =n^2\left\vert\frac{l_{n+2}-2l_{n+1}+l_{n}+l_{n}(l_{n+1}-l_{n+2})+l_{n+2}(l_{n+1}-l_{n})}{(1-l_{n+2})(1-l_{n+1})(1-l_{n})}\right\vert\\
&\le (\sup \vert M_n\vert )^3\, \left[(1+\big\vert l_{n}\big\vert)\, n^2\, \big\vert l_{n+2}-l_{n+1}\big\vert +(1+\big\vert l_{n+2}\big\vert)\, n^2\, \big\vert l_{n+1}-l_{n}\big\vert\right]\\
&\le c \sup_n \left( n^2\,\big\vert l_{n+1}-l_{n}\big\vert\right) <\infty\ .
\end{split}
\eqnn
Consequently, Lemma~\ref{AB} implies $\big(1-\mathcal{R}(0)\big)^{-1}\in\mathcal{L}\big(C^{1+\delta}(\sph)\big)$.
But then $\big(1-\mathcal{R}(0)\big)^{-1}$ is also a bounded operator on
$$
H^{s}(\sph) =\left\{h\in L_2(\sph)\,;\, \| h\|_{H^{s}}:=\sum_{n\in\Z} (1+n^2)^{s} \vert \hat{h}_n\vert^2 <\infty \right\}
$$
for each $s>0$ due to \eqref{39}. Recalling that $H^{s}(\sph)$ {is densely 
embedded} in $h^{1+\delta}(\sph)$ provided $s>5/2$, we deduce that 
$\big(1-\mathcal{R}(0)\big)^{-1}\in\mathcal{L}\big(h^{1+\delta}(\sph)\big)$.
\end{proof}

According to Lemma~\ref{L14} and \eqref{P3a}, we are led to look for solutions 
$$
\rho\in C\big([0,T],\mathcal{W}\big)\cap C^1\big([0,T],h^{1+\delta}(\sph)\big)
$$
to the fully nonlinear equation
\bqn\label{41}
\partial_t\rho =\big(1-\mathcal{R}(\rho)\big)^{-1} \mathcal{B}_\oo (\rho) T\big(\rho\,,\,-\mathcal{K}(\rho)\big)=: F(\rho)\, ,\quad t\in (0,T]\ .
\eqn
{The following proposition is instrumental in the computation of the 
linearization in zero of this evolution equation.}

\begin{prop}\label{L15}
$F\in C^\infty\big(\mathcal{W},h^{1+\delta}(\sph)\big)$ and, for $h\in h^{4+\delta}(\sph)$, 
$$
\partial F(0)[h]= \big(1-\mathcal{R}(0)\big)^{-1} \mathcal{B}_\mathrm{o} (0) T\big(0\,,\,(2(\gamma_\mathrm{i}-\gamma_\mathrm{o})+\sigma) h+\sigma\ddot{h}\big)\ .
$$
In particular, if $h\in h^{4+\delta}(\sph)$ with $h(\theta)=\sum_{n\in\Z} \hat{h}_n e^{in\theta}$, then
$$
\partial F(0)[h](\theta)=\sum_{n\in\Z\setminus\{0\}}  \hat{h}_n q_n e^{in\theta}
$$
where, for $n\in\Z\setminus\{0\}$,
\bqn\label{44}
q_n:=\frac{A_n+{i\,}\mathrm{sign}(n) B}{A_n^2+B^2}\mu(n)
\eqn
with
$$
A_n:=\mathrm{sign}(n)\frac{R^{2n}+1}{R^{2n}-1}\alpha_\mathrm{o}+\alpha_\mathrm{i}\, ,\quad B:=\beta_\mathrm{o}-\beta_\mathrm{i}\, ,\quad  \mu(n):=\vert n\vert \big(\sigma-2(\gamma_\mathrm{o}-\gamma_\mathrm{i})-\sigma n^2\big)\ .
$$
\end{prop}

\begin{proof}
Smoothness of the map $F$ {follows from} Lemma~\ref{L10}, Proposition~\ref{P12}, 
Proposition~\ref{P13}, and Lemma~\ref{L14}. Let $h\in h^{4+\delta}(\sph)$. Then
\bqnn
\begin{split}
\partial F(0)[h] =\ & \partial \big(1-\mathcal{R}(\cdot)\big)^{-1}(0)[h]\,  \mathcal{B}_\mathrm{o} (0)\, T\big(0\,,\,-\mathcal{K}(0)\big) +
\big(1-\mathcal{R}(0)\big)^{-1} \,\partial \mathcal{B}_\mathrm{o} (0)[h]\, T\big(0\,,\,-\mathcal{K}(0)\big)\\
& +\big(1-\mathcal{R}(0)\big)^{-1} \,\mathcal{B}_\mathrm{o} (0) \, \partial_\rho T\big(0\,,\,-\mathcal{K}(0)\big)[h]
+ \big(1-\mathcal{R}(0)\big)^{-1} \,\mathcal{B}_\mathrm{o} (0)\, T\big(0\,,\,-\partial\mathcal{K}(0)[h]\big)\ .
\end{split}
\eqnn
But $-\mathcal{K}(0)=-\sigma-\gamma_\mathrm{o}+\gamma_\mathrm{i}=:c\in\R$ and $T(0,c)=c$ by uniqueness, so $\partial \mathcal{B}_\mathrm{o} (0)[h]=0$ as this is the derivative of $\big(\rho\mapsto \mathcal{B}_\mathrm{o} (\rho)c=0\big)$ and similarly $ \partial_\rho T\big(0\,,\,-\mathcal{K}(0)\big)[h]=0$ as this is the derivative of \mbox{$\big(\rho\mapsto T(\rho,c)=c\big)$}. The formula for $\partial F(0)[h]$ now follows from Lemma~\ref{L10}. To compute its Fourier expansion, consider \mbox{$h\in h^{4+\delta}(\sph)$} with $h(\theta)=\sum_{n\in\Z} \hat{h}_n e^{in\theta}$.
Invoking \eqref{Qo0} and recalling that $ \mathcal{B}_\mathrm{o} (0)=-\big(\alpha_\oo\partial_r-\beta_\oo\partial_\theta)/\vert\Theta_\oo\vert^2$ on $\sph$,  we obtain
\bqnn
\begin{split}
\partial F(0)[h](\theta)&= \big(1-\mathcal{R}(0)\big)^{-1} \mathcal{B}_\mathrm{o} (0) T\big(0\,,\,(2(\gamma_\mathrm{i}-\gamma_\mathrm{o})+\sigma) h+\sigma\ddot{h}\big) (\theta)\\
&= \sum_{n\in\Z\setminus\{0\}} \hat{h}_n\, \frac{n \big(2(\gamma_\mathrm{i}-\gamma_\mathrm{o})+\sigma-\sigma n^2\big) (R^{2n}-1)}{\Theta_\oo+R^{2n}\bar{\Theta}_\oo}\,\frac{1}{1-l_n}\, e^{in\theta}
\end{split}
\eqnn
with $l_n$ given by \eqref{36A}. Elementary calculations {now} lead to the 
assertion.
\end{proof}

Observe that $A_n=A_{-n}$ for $n\in\Z\setminus\{0\}$ and that \
\bqn\label{45}
A_n\searrow \alpha_\oo+\alpha_\ii\quad\text{as}\quad \vert n\vert \rightarrow \infty\ .
\eqn
The next proposition is a consequence of the previous lemma and fundamental for our well-posedness result.

\begin{prop}\label{P17}
$-\partial F(0)\in \mathcal{H}\big(h^{4+\delta}(\sph),h^{1+\delta}(\sph)\big)$, that is, $\partial F(0)\in \mathcal{L}\big(h^{4+\delta}(\sph),h^{1+\delta}(\sph)\big)$ is the generator of an analytic semigroup on $h^{1+\delta}(\sph)$.
\end{prop}

\begin{proof}
Based on Proposition~\ref{AB}, we prove in a first step that $\big(\lambda-\partial F(0)\big)^{-1}\in\mathcal{L}\big(h^{1+\delta}(\sph),h^{4+\delta}(\sph)\big)$,
where, according to Proposition~\ref{L15} (with $q_0:=0$){, we have that}
$$
\big(\lambda-\partial F(0)\big)^{-1} h =\sum_{n\in\Z}\hat{h}_n\,\frac{1}{\lambda-q_n}\, e^{in\theta}\, ,\quad \mathrm{Re}\, \lambda\ge \lambda_*\, ,
$$
for $h(\theta)=\sum_{n\in\Z} \hat{h}_n e^{in\theta}$ and,
by \eqref{45},
$$
\lambda_* :=1+\frac{2\vert\gamma_\oo-\gamma_\ii\vert}{\alpha_\oo+\alpha_\ii}>\mathrm{Re}\, q_n\, ,\quad n\in\Z\ .
$$
Let $\lambda\in\C$ with $ \mathrm{Re}\, \lambda\ge \lambda_*$ be fixed and set $Q_n^\lambda:=(\lambda-q_n)^{-1}$. Then
\bqn\label{47}
\sup_{n\in\Z\setminus\{0\}}\vert n\vert^3\left\vert Q_n^\lambda\right\vert <\infty\, ,
\eqn
since
$$
\frac{q_n}{n^3}\longrightarrow \pm\frac{\alpha_\oo+\alpha_\ii\pm iB}{(\alpha_\oo+\alpha_\ii)^2+B^2}\,\sigma\quad\text{as}\quad n\longrightarrow\pm \infty\ .
$$
Next observe that, for each $m\in\N$,
\bqnn\label{48}
\vert n\vert^m\vert A_{n+1}-A_n\vert =\vert n\vert^m \frac{2R^{2n}(R-1)}{(R^{2n}-1)(R^{2n+1}-1)}\longrightarrow 0 \quad\text{as}\quad n\longrightarrow\pm \infty\ .
\eqnn
Thus, letting $z_n:=q_n/\mu(n)$ we derive from \eqref{45} 
\bqnn
\vert n\vert^m\vert z_{n+1}-z_n\vert\le \frac{A_{n+1}A_n}{(A_{n+1}^2+B^2)(A_n^2+B^2)}\,\vert n\vert^m\,\vert A_{n+1}-A_n\vert +\vert n\vert^m\, \vert A_{n+1}-A_n\vert\,  B^2\, ,
\eqnn
so, for each $m\in\N$,
\bqn\label{49}
\vert n\vert^m\vert z_{n+1}-z_n\vert  \longrightarrow 0 \quad\text{as}\quad \vert n\vert\longrightarrow \infty\ .
\eqn
Thus, \eqref{47} gives
\bqnn
\begin{split}
\vert n\vert^4\, \left\vert Q_{n+1}^\lambda -Q_n^\lambda\right\vert 
& =\left\vert n^3 Q_{n+1}^\lambda\right\vert\, \left\vert n^3 Q_{n}^\lambda\right\vert\, \left\vert\frac{q_{n+1}-q_n}{n^2}\right\vert\\
&\le  c \frac{\vert z_{n+1}\vert\,\vert \mu(n+1)-\mu(n)\vert}{n^2} +c \vert z_{n+1}-z_n\vert\, \vert n\vert\, \left\vert\frac{\mu(n)}{n^3}\right\vert \ .
\end{split}
\eqnn 
Taking into account that $(z_n)$ is bounded and that $\vert\mu(n)/n^3\vert\rightarrow \sigma$ as $\vert n\vert\rightarrow\infty$, we conclude
\bqn\label{50}
\sup_{n\in\Z\setminus\{0\}}\vert n\vert^4\left \vert Q_{n+1}^\lambda - Q_n^\lambda\right\vert <\infty\ .
\eqn
In particular, we have shown
\bqn\label{51}
\sup_{n\in\Z\setminus\{0\}} n^{-2}\,\left\vert q_{n+1} - q_n\right\vert <\infty\ .
\eqn
Observe then that
$${
\mu(n)\big[\mu(n+1)-\mu(n+2)\big]+\mu(n+2)\big[\mu(n+1)-\mu(n)\big]=O(n^4)\, ,}
$$
hence, due to \eqref{49},
\bqn\label{53}
\begin{split}
&\left\vert\frac{1}{n^4}{\Big[}q_n\left(q_{n+1}-q_{n+2}\right)+q_{n+2}
\left(q_{n+1}-q_n\right){\Big]}\right\vert
\\
& \qquad\quad\quad \le \left\vert z_n\, z_{n+1}\,      \frac{\mu(n)\big[\mu(n+1)-\mu(n+2)\big]+\mu(n+2)\big[\mu(n+1)-\mu(n)\big]}{n^4}\right\vert\\
&\qquad\qquad\quad + \left\vert z_n\,\frac{\mu(n)}{n^3}\,\frac{\mu(n+1)}{n^3}\, (z_{n+1}-z_{n+2})\, n^2\right\vert + \left\vert z_n\, \frac{\mu(n)}{n^3}\,\frac{\mu(n+2)}{n^3}\, (z_{n+1}-z_{n+2})\, n^2\right\vert\\
&\qquad\qquad\quad  + \left\vert z_{n+1}\, \frac{\mu(n+2)}{n^3}\, \frac{\mu(n+1)}{n^3}\, (z_{n+2}-z_n)\, n^2\right\vert\\
& \quad\qquad\quad \le c\ .
\end{split}
\eqn
Writing
\bqnn
\begin{split}
&\vert n\vert^5\, \left\vert Q_{n+2}^\lambda-2 Q_{n+1}^\lambda +Q_n^\lambda\right\vert \\
&\qquad =\left\vert n^3 Q_{n+2}^\lambda\right\vert\, \left\vert n^3 Q_{n+1}^\lambda\right\vert\, \left\vert n^3 Q_{n}^\lambda\right\vert\, \left\vert\frac{\lambda (q_{n+2}-2 q_{n+1}+q_n)}{n^4}+\frac{q_n(q_{n+1}-q_{n+2})+q_{n+2}(q_{n+1}-q_n)}{n^4}\right\vert
\end{split}
\eqnn 
we deduce from \eqref{47} and \eqref{51} {that}
\bqn\label{54}
\sup_{n\in\Z\setminus\{0\}}\vert n\vert^5\, \left\vert Q_{n+2}^\lambda-2 Q_{n+1}^\lambda +Q_n^\lambda\right\vert  <\infty\ .
\eqn
Consequently, $\big(\lambda-\partial F(0)\big)^{-1}\in\mathcal{L}\big(C^{1+\delta}(\sph),C^{4+\delta}(\sph)\big)$ for $\mathrm{Re}\,\lambda\ge \lambda_*$ by Lemma~\ref{AB} and \eqref{47}, \eqref{50}, and \eqref{54}. Since \eqref{47} also ensures 
\bqn\label{101}
\big(\lambda-\partial F(0)\big)^{-1}\in\mathcal{L}\big(H^s(\sph)\big)\, ,\quad s>0\, ,
\eqn
we conclude $\big(\lambda-\partial F(0)\big)^{-1}\in\mathcal{L}\big(h^{1+\delta}(\sph),h^{4+\delta}(\sph)\big)$ for \mbox{$\mathrm{Re}\,\lambda\ge \lambda_*$} as in the proof of Lemma~\ref{L14}.

The second step consists of proving the resolvent estimate
\bqn\label{100}
\vert\lambda\vert\,\left\| \big(\lambda-\partial F(0)\big)^{-1}\right\|_{\mathcal{L}(h^{1+\delta}(\sph))}\,\le\, c\, ,\quad \mathrm{Re}\,\lambda\ge \lambda_*\ .
\eqn
Since $\mathrm{Re}\, q_n <0$ for $\vert n\vert$ sufficiently large, elementary calculations show that 
\bqn\label{55}
\vert\lambda-q_n\vert^2\ge c_0\vert \lambda\vert^2\, ,\quad \mathrm{Re}\, \lambda\ge\lambda_*\, ,\quad n\in\Z\, ,
\eqn
for some $c_0>0$. Thus, setting $S_n^\lambda:=\lambda (\lambda-q_n)^{-1}$it follows
\bqn\label{56}
\sup_{n\in\Z\setminus\{0\}\, ,\,  \mathrm{Re}\, \lambda\ge\lambda_*}\, \vert S_n^\lambda\vert <\infty\ .
\eqn
Similarly, there is $c_1>0$ such that
\bqn\label{57}
\vert\lambda-q_n\vert^2\ge c_1\vert q_n\vert^2\, ,\quad \mathrm{Re}\, \lambda\ge\lambda_*\, ,\quad n\in\Z\, ,
\eqn
and we thus obtain from
$$
\vert n\vert\,\left\vert S_{n+1}^\lambda-S_{n}^\lambda\right\vert =\left\vert S_{n+1}^\lambda\right\vert\, \left\vert \frac{n^3}{\lambda-q_n}\right\vert\, \left\vert\frac{q_{n+1}-q_n}{n^2}\right\vert
$$
together with \eqref{51}, \eqref{56}, and \eqref{57}  combined with \eqref{47} that
\bqn\label{58}
\sup_{n\in\Z\setminus\{0\}\, ,\,  \mathrm{Re}\, \lambda\ge\lambda_*}\, \vert n\vert \left\vert S_{n+1}^\lambda-S_n^\lambda\right\vert <\infty\ .
\eqn
Finally, noticing that the right hand side of
\bqnn
\begin{split}
&\left\vert\frac{q_{n+2}-2 q_{n+1}+q_n}{n}\right\vert\\
&\qquad \le \vert z_{n+2}\vert\, \left\vert\frac{\mu(n+2)-2 \mu(n+1)+\mu(n)}{n}\right\vert + 2\left\vert\frac{\mu(n+1)}{n^3}\right\vert\, \vert z_{n+2}-z_{n+1}\vert\, n^2 +\left\vert\frac{\mu(n)}{n^3}\right\vert\,\vert z_n-z_{n+1}\vert\, n^2
\end{split}
\eqnn
is bounded by \eqref{49} and writing
\bqnn
\begin{split}
n^2\,&\left\vert S_{n+2}^\lambda-2S_{n+1}^\lambda+S_{n+2}^\lambda\right\vert\\
& =\left\vert S_{n+2}^\lambda\right\vert\, \left\vert n^3 Q_{n+1}^\lambda\right\vert\, \left\vert \frac{n^3}{\lambda-q_n}\right\vert\, \frac{1}{n^4}\, \left\vert \lambda (q_{n+2}-2q_{n+1}+q_n)+q_n(q_{n+1}-q_{n+2})+q_{n+2}(q_{n+1}-q_n)\right\vert
\end{split}
\eqnn
we deduce from \eqref{47}, \eqref{53}, \eqref{55}, and \eqref{56}
\bqn\label{60}
\sup_{n\in\Z\setminus\{0\}\, ,\,  \mathrm{Re}\, \lambda\ge\lambda_*}\, n^2\,\left\vert S_{n+2}^\lambda-2S_{n+1}^\lambda+S_n^\lambda\right\vert <\infty\ .
\eqn
Therefore, Lemma~\ref{AB} and \eqref{56}, \eqref{58}, and \eqref{60} imply
$$
\vert\lambda\vert\,\left\| \big(\lambda-\partial F(0)\big)^{-1}\right\|_{\mathcal{L}(C^{1+\delta}(\sph)}\,\le\, c\, ,\quad \mathrm{Re}\, \lambda\ge \lambda_*\, ,
$$
whence \eqref{100} due to \eqref{101}. This proves the assertion.
\end{proof}

Now we are in a position to establish {a} well-posedness result regarding 
equation~\eqref{41}.

\begin{thm}\label{ex}
There exists an open zero neighborhood $\mathcal{O}\subset\mathcal{V}$ in $h^{4+\delta}(\sph)$ such that for each $\rho_0\in\mathcal{O}$ there is $T:=T(\rho_0)>0$  and a unique solution
$$
\rho\in C\big([0,T],\mathcal{V}\big)\cap C^1\big([0,T],h^{1+\delta}(\sph)\big)
$$
to
$$
 \partial_t\rho=F(\rho)\, , \quad t>0\, ,\qquad 
 \rho(0)=\rho_0\ .
$$
Moreover, $\rho\big([0,T])\big)\subset\mathcal{O}$.
\end{thm}

\begin{proof}
We shall invoke \cite[Thm.8.1]{Lunardi}. Fix $\xi\in (0,\delta)$ and put $\vartheta:=(\delta-\xi)/3$. Set 
$$
E:=h^{1+\xi}(\sph)\, ,\quad E_0:=h^{1+\delta}(\sph)\, ,\quad E_1:=h^{4+\delta}(\sph)
$$ 
in \cite[Thm.8.1]{Lunardi}. As $\delta\in (0,1)$ was arbitrary in Proposition~\ref{P17}, it follows that
$$
-\partial F(0)\in \mathcal{H}\big(h^{4+\xi}(\sph),h^{1+\xi}(\sph)\big)\ .
$$
Thus, since $\mathcal{H}\big(h^{4+\xi}(\sph),h^{1+\xi}(\sph)\big)$ is open in 
$\mathcal{L}\big(h^{4+\xi}(\sph),h^{1+\xi}(\sph)\big)$, there is an open zero 
neighborhood $U_\xi$ in $h^{4+\xi}(\sph)$ such that $-\partial F(\rho)\in 
\mathcal{H}\big(h^{4+\xi}(\sph),h^{1+\xi}(\sph)\big)$ for each $\rho\in U_\xi$. 
Then $\mathcal{O}:=U_\xi\cap \mathcal{W}$ with $\mathcal{W}$ from Lemma~\ref{L14} 
is an open zero neighborhood in $h^{4+\delta}(\sph)$. Furthermore, 
$\partial F(\rho):h^{4+\delta}(\sph)\rightarrow h^{1+\delta}(\sph)$ for 
$\rho\in\mathcal{O}$ is the part of $-\partial F(\rho)\in \mathcal{H}
\big(h^{4+\xi}(\sph),h^{1+\xi}(\sph)\big)$ in 
$$
 h^{1+\delta}(\sph)\doteq \big(h^{1+\xi}(\sph),h^{4+\xi}(\sph)\big)_{\vartheta,\infty}^0
$$ 
with continuous interpolation functor $(\cdot,\cdot)_{\vartheta,\infty}^0$ and 
$$
\big\{h\in h^{4+\xi}(\sph)\,;\, \partial F(\rho)[h]\in h^{1+\delta}(\sph)\big\} =h^{4+\delta}(\sph)\ .
$$ 
Now the assertion is a consequence of \cite[Thm.8.1]{Lunardi}.
\end{proof}

To finish off the proof of Theorem~\ref{T1} let
$$
\rho\in C\big([0,T],{\mathcal{V}}\big)\cap C^1\big([0,T],h^{1+\delta}(\sph)\big)
$$
be the solution to \eqref{41} for a given initial value $\rho_0\in\mathcal{O}$. Then
$$
Q_\ii:=S\big(\rho,\partial_t\rho\big)\in C\big([0,T],h^{2+\delta}(B^\ii)\big)\, ,\qquad Q_\oo:=T\big(\rho,Q_\ii-\mathcal{K}(\rho)\big)\in C\big([0,T],h^{2+\delta}(B^\oo)\big)
$$
by Lemma~\ref{L10}, Proposition~\ref{P12}, and Proposition~\ref{P13}. Since $\rho$ solves \eqref{P3}, it follows, for $\rho=\rho(t)$ with $t\in [0,T]$ fixed, that 
$$
 \frac{\phi_*^\rho \partial_t\rho }{\vert \nabla N_\rho\vert}\, =\, 
 \frac{-1}{\vert\Theta_\oo\vert^2}\Big(\alpha_\oo\partial_{\nu_\rho}
 \big(\phi_*^\rho Q_\oo\big)+\beta_\oo\partial_{\tau_\rho}\big(\phi_*^\rho 
 Q_\oo\big)\Big)\quad \text{on}\quad \Gamma_{\rho}\ .
$$
Recalling from \eqref{P2a} that $\Delta \big(\phi_*^\rho Q_\oo\big)=0$ and $\mathrm{div}\left(\vec{z}\times\nabla \big(\phi_*^\rho Q_\oo\big)\right)=0$ in $\Om_\rho^\oo$, we deduce from Gauss' Theorem
\bqnn
\frac{-1}{\vert\Theta_\oo\vert^2} \int_{\Gamma_\rho} \left(\alpha_\oo\partial_{\nu_\rho}\big(\phi_*^\rho Q_\oo\big)+\beta_\oo\partial_{\tau_\rho}\big(\phi_*^\rho Q_\oo\big)\right)\, \rd\sigma = \frac{-1}{\vert\Theta_\oo\vert^2} \int_{R\sph} \left(\alpha_\oo\partial_\nu\big(\phi_*^\rho Q_\oo\big)+\beta_\oo\partial_\tau\big(\phi_*^\rho Q_\oo\big)\right)\, \rd\sigma
\eqnn
and thus, due to $\phi_*^\rho Q_\oo=Q_\oo$ and \eqref{P2c} on {$R\, \sph$, that}
\bqn\label{just}
\int_{\Gamma_\rho} \frac{\phi_*^\rho \partial_t\rho }{\vert \nabla N_\rho\vert}\, \rd\sigma =0\, ,\quad t\in [0,T]\ .
\eqn
Consequently, with $$P_\ii(t):=\phi_*^{\rho(t)} Q_\ii(t)\in h^{2+\delta}(\Omega_{\rho(t)}^\ii)\, ,\qquad P_\oo(t):=\phi_*^{\rho(t)} Q_\oo(t)h^{2+\delta}(\Omega_{\rho(t)}^\oo)$$ for $t\in [0,T]$ we obtain a solution $(\rho,P_\ii,P_\oo)$ to \eqref{15a}-\eqref{15d} which is unique up to additive constants in the pressures $P_\ii$ and $P_\oo$. This yields Theorem~\ref{T1}.

\section{Proof of Theorem~\ref{T2}}\label{stable}

We first prove instability of the trivial solution if $\varrho_\ii>\varrho_\oo$ as claimed in Theorem~\ref{T2}. {Recall that \mbox{$\gamma_j:=\varrho_j\omega^2/2$}.}

\begin{thm}\label{T21}
If $\varrho_\mathrm{i}>\varrho_\mathrm{o}$, then
$$
\partial_t\rho=F(\rho)\, ,\quad t>0\, ,\qquad \rho(0)=\rho_0\, ,
$$
has backward solutions {which do exponentially decay to zero}. In particular, 
the trivial solution $\rho=0$ of this flow is unstable.
\end{thm}

\begin{proof}
The compact embedding $h^{4+\delta}(\sph)\hookrightarrow h^{1+\delta}(\sph)$ and 
Proposition~\ref{P17} imply that $\partial F(0)$ has compact resolvent. So the 
spectrum of $\partial F(0)$ consists of eigenvalues only, which, according to 
Proposition~\ref{L15}, are given by {$\big\{q_n\,;\,n\in\Z\setminus
{\{0\}}\big\}$ {and} $q_0:=0$}. {Since 
$\varrho_\mathrm{i}>\varrho_\mathrm{o}$, \eqref{44} implies that} 
$\mu(1)>0$, hence \eqref{45} shows {that} \mbox{$\mathrm{Re}\, q_1 >0$}. 
Clearly,
$$
\inf\left\{\mathrm{Re}\, q_n\,;\, \mathrm{Re}\, q_n>0\right\} \,>\, 0\ .
$$
The assertion now follows from \cite[Thm.9.1.3]{Lunardi}.
\end{proof}

To prove stability of the trivial solution if $\varrho_\oo>\varrho_\ii$, we need an auxiliary result.

\begin{lem}\label{L22}
{Let $\mathcal{W}$ be given as in Lemma~\ref{L14}. Then
$$
\int_{\sph} (1+\rho) F(\rho)\,\rd \sigma=0\, ,\quad \rho\in\mathcal{W}\, .
$$}
\end{lem}

\begin{proof}
Fix  $\rho\in\mathcal{W}$ and set
$$
h_{0,\rho}^{1+\delta}(\sph):=\left\{f\in h^{1+\delta}(\sph)\,;\,  \int_{\sph} (1+\rho) f\,\rd \sigma=0\right\}\ .
$$
We claim that $1-\mathcal{R}(\rho)$ is an isomorphism on $h_{0,\rho}^{1+\delta}(\sph)$. To see this, set $T_f:=T\big(\rho,S(\rho,f)\big)$ for $f\in h^{1+\delta}(\sph)$. Then, as in \eqref{just},
$$
\int_{\sph} (1+\rho) \mathcal{R}(\rho) f\,\rd \sigma =\int_{\sph} (1+\rho) \mathcal{B}_\oo(\rho) T_f\,\rd \sigma= 0
$$
and Lemma~\ref{L14} implies that $1-\mathcal{R}(\rho)$ is indeed an isomorphism on $h_{0,\rho}^{1+\delta}(\sph)$. But, as above, $$\mathcal{B}_\oo(\rho)T\big(\rho,-\mathcal{K}(\rho)\big)\in h_{0,\rho}^{1+\delta}(\sph)\ .$$ Therefore, 
$$
F(\rho)=\big(1-\mathcal{R}(\rho)\big)^{-1} \mathcal{B}_\oo (\rho) T\big(\rho\,,\,-\mathcal{K}(\rho)\big)\in  h_{0,\rho}^{1+\delta}(\sph)\ .
$$
\end{proof}

{We conclude the proof of Theorem~\ref{T2} by stating the stability result 
for which we need to define
$$
 h_{0}^{s}(\sph):=\big\{f\in h^{s}(\sph)\,;\,  \int_{\sph} f\,\rd \sigma=0\big\}\, ,
$$
for $s>0$.}
\begin{thm}\label{T23}
{If $\varrho_\mathrm{o}>\varrho_\mathrm{i}$,} the trivial solution $\rho=0$ of 
$$
\partial_t\rho=F(\rho)\, ,\quad t>0\, ,
$$
is stable. More precisely, there are numbers $\omega, r, M>0$ such that for each initial datum \mbox{$\rho_0\in  h_{0}^{1+\delta}(\sph)$} with $\|\rho_0\|_{h^{1+\delta}(\sph)}\le r$  there is a unique global solution
$$
\rho\in C\big(\R^+,h_{0}^{4+\delta}(\sph)\big)\cap C^1\big(\R^+,h_{0}^{1+\delta}(\sph)\big)
$$
{with $\rho(0)=\rho_0$} and
$$
\left\|\rho(t)\left(1+\frac{\rho(t)}{2}\right)\right\|_{h^{4+\delta}(\sph)}+ \left\|\dot{\rho}(t)\left(1+\rho(t)\right)\right\|_{h^{1+\delta}(\sph)}\ \le\ M\,e^{-\omega t}\, \left\|\rho_0\left(1+\frac{\rho_0}{2}\right)\right\|_{h^{4+\delta}(\sph)}\, ,\quad t\ge 0\ .
$$
\end{thm}

\begin{proof}
Letting $\zeta:=\rho+\rho^2/2$ for $\rho\in \mathcal{W}$, {problem}
$\partial_t\rho=F(\rho)$, $ t>0$,
is equivalent to
\bqn\label{GG}
\partial_t\zeta=G(\zeta)\, ,\: t>0\, ,
\eqn
where $G\in C^2\big(Z,h^{1+\delta}(\sph)\big)$ with $Z:=\big\{{\zeta=}\rho+\rho^2/2\,;\, \rho\in\mathcal{W}\big\}$ is given by 
$$
G(\zeta):=\big(2\sqrt{1+\zeta}-1\big) F\big(2\sqrt{1+\zeta}-1\big)\ .$$
Moreover, Lemma~\ref{L22} implies
$G\in C^2\big(Z_0,h_0^{1+\delta}(\sph)\big)$ for $Z_0:=Z\cap h_0^{4+\delta}(\sph)$.
Also, $$\partial G(0)=\partial F(0)\in\mathcal{L}\big(h_0^{4+\delta}(\sph),h_0^{1+\delta}(\sph)\big)\ .$$ Thus $\partial G(0)$ has compact resolvent and its (point) 
spectrum equals ${\big\{}q_n\,;\, n\in\Z\setminus\{0\}{\big\}}$. 
Now, by \eqref{44},
$$
\mathrm{Re}\, q_n=\frac{A_n}{A_n^2+B^2}\,\mu(n)\, ,\quad n\in\Z\setminus\{0\}\, ,
$$
{while $\varrho_\mathrm{o}>\varrho_\mathrm{i}$ implies that}
$$
\mu(n)=\vert n\vert \big(\sigma-2(\gamma_\mathrm{o}-\gamma_\mathrm{i})-\sigma n^2\big)\le -2(\gamma_\oo-\gamma_\ii)<0\, ,\quad n\in\Z\setminus\{0\}\ .
$$
{This, combined with \eqref{45},} shows that the spectrum of $\partial G(0)$ 
is contained in a half plane \mbox{$[\mathrm{Re}\, \lambda\le -\omega]$} for some 
$\omega>0$. The assertion now follows from \cite[Thm.9.1.2]{Lunardi} applied to 
\eqref{GG}.
\end{proof}

{Note that our analysis {yields an explicit estimate of} the exponential 
decay {rate} $\omega$ in terms of the physical parameters through \eqref{44}.}


\end{document}